\documentclass[10pt]{amsart}
\usepackage{palatino, amssymb, amsfonts, latexsym, amsmath, amscd}

\usepackage{pst-all}
\usepackage{amsfonts}
\usepackage{amsthm}

\usepackage[latin1]{inputenc}

\usepackage{graphics}

\DeclareMathAlphabet{\mathpzc}{OT1}{pzc}{m}{it}
\input xy
\xyoption{all}

\newcommand{\U}{\mathbf{U}}

\newcommand{\Oint}{\mathcal O_{\text{int}}}

\newcommand{\nc}{\newcommand}
\nc{\g}{{\mathfrak g}}
\nc{\ghat}{\widehat{\g}}
\nc{\mc}{\mathcal}
\nc{\ep}{\epsilon}
\nc{\la}{\lambda}
\nc{\Z}{{\mathbb Z}}
\nc{\C}{{\mathbb C}}
\nc{\on}{\operatorname}
\nc{\wt}{\widetilde}
\nc{\om}{\omega}
\nc{\ol}{\overline}

\newtheorem{theorem}{Theorem}[section]
\newtheorem{lemma}[theorem]{Lemma}
\newtheorem{prop}[theorem]{Proposition}
\newtheorem{cor}[theorem]{Corollary}
\theoremstyle{definition}
\newtheorem{definition}[theorem]{Definition}
\newtheorem{example}[theorem]{Example}

\theoremstyle{remark}
\newtheorem{remark}[theorem]{Remark}

\newcommand{\A}{\mathcal A}

\newcommand{\Ud}{\dot{\mathbf U}}
\newcommand{\Uda}{_{\mathcal A}\dot{\mathbf U}}

\newcommand{\fA}{_\mathcal A\mathbf f}
\newcommand{\Lieg}{\mathfrak g}

\numberwithin{equation}{section}

\begin{document}

\title{Langlands duality for representations and quantum groups at a root of unity.}

\author[Kevin Mcgerty]{Kevin Mcgerty$^1$}\thanks{$^1$Supported by a Royal Society University Research Fellowship.}

\address{Department of Mathematics, Imperial College London}

\date{February, 2009}

\begin{abstract}
We give a representation-theoretic interpretation of the Langlands character duality of \cite{FH}, and show that the ``Langlands branching multiplicites'' for symmetrizable Kac-Moody Lie algebras are equal to certain tensor product multiplicities. For finite type quantum groups, the connection with tensor products can be explained in terms of tilting modules.
\end{abstract}

\maketitle

\section{Introduction}

Let $\Lieg$ be a simple Lie algebra and $^L\Lieg$ its Langlands dual Lie algebra. In \cite{FH} a duality between the irreducible characters of $\Lieg$ and $^L\Lieg$ was established and a number of conjectures were made about its properties, both on the level of representations and, more combinatorially, at the level of crystals. In this paper we generalize the character duality to the category $\Oint$ of integrable representations in category $\mathcal O$ for a symmetrizable Kac-Moody Lie algebra, and establish, in this more general context, a number of the conjectures of \cite{FH}. With a mild restriction on the generalized Cartan matrix we also give a representation-theoretic interpretation of the character duality using Lusztig's modified quantum groups at a root of unity.

It will be convenient to use Lusztig's notion of a root datum and Cartan datum, the latter being essentially a symmetrizable generalized Cartan matrix with an integral choice of symmetrization. Indeed for any symmetrizable generalized Cartan matrix $C = (a_{ij})_{i,j \in I}$ on an indexing set $I$, we may choose a root datum $(X,Y,I)$ consisting of a weight lattice $X$, a coweight lattice $Y$, a perfect pairing $\langle\cdot,\cdot \rangle\colon Y \times X \to \mathbb Z$, along with a set of simple roots $\{\alpha_i\}_{i \in I}\subset X$ and a set of simple coroots $\{\check{\alpha}_i\}_{i \in I}\subset Y$, which satisfy $\langle \check{\alpha}_i, \alpha_j \rangle = a_{ij}$. We will assume that both the simple roots and simple coroots are linearly independent, so that we have the standard partial ordering on $X$, and a dominant cone 
\[
X^+ = \{\lambda \in X: \langle \check{\alpha}_i,\lambda \rangle  \geq 0, \forall i \in I \}.
\]

Let $(d_i)_{i \in I}$ be an integral vector such that $DC$ is symmetric, where $D$ is the diagonal matrix with $D_{ii} = d_i$. Set $d$ to be the least common multiple of the $d_i$, and let $l_i = d/d_i$. Let $^L\mathfrak g$ be the Langlands dual Kac-Mpody algebra with Cartan matrix $C^t$. Then we may embed the weight lattice of $^L\Lieg$ into $X$ in such a way that $\varpi_i^L \mapsto \varpi_i^* = l_i \varpi_i$ (where $\varpi_i, \varpi_i^L$ are the fundamental weights). Let $X^*$ denote the image of this map\footnote{For Kac-Moody algebras of finite type, this corresponds to the embedding used by Frenkel and Hernandez, who denote its image by $P'$.}.

Given a simple highest weight representation $\nabla(\lambda)$ of $\Lieg$, such that $\lambda \in X^*\cap X^+$, let $c^*(\nabla(\lambda))$ denote the direct sum of those weight spaces of $\nabla(\lambda)$ whose weights lie in $X^*$. In \cite{FH} it was shown that the character of $c^*(\nabla(\lambda))$ is the virtual character of a representation of $^L\Lieg$: more precisely it was shown that
\begin{equation}
\label{characterdualityformula}
\chi(c^*(\nabla(\lambda))) = \chi^L(\lambda) + \sum_{\mu \in X^*, \mu < \lambda} m_{\mu}^\lambda \chi^L(\mu), \quad (m_{\mu}^\lambda \in \mathbb Z),
\end{equation}
where $\chi^L(\mu)$ is the character of the simple highest module $\nabla^L(\mu)$ for $^L\Lieg$ of highest weight $\mu$ (in the notation of that paper, the left-hand side is written $\Pi(\chi(\lambda))$). Moreover, it was also shown there that the crystal graph for $\nabla(\lambda)$ naturally contains the crystal graph of $\nabla^L(\lambda)$ (this result is also implied by a previous result of Kashiwara \cite{K96} in a somewhat different context). Thus the character $\chi^L(\lambda)$ can be viewed as a ``subcharacter'' of the character of $\chi(\lambda) = \chi(\nabla(\lambda))$ (that is, the dimension of a weight space in $\nabla^L(\lambda)$ is at most the dimension of the corresponding weight space of $\nabla(\lambda)$). 

In \cite{FH} the authors proposed an interpretation of the character duality $c^*$ in terms of representations of a two-parameter deformation of $\Lieg$. While we cannot establish the duality in this context, we have obtained an interpretation in terms of the representation theory of (one-parameter) quantum groups at a root of unity. Since the strategy using two-parameter deformations also involved specializing one parameter to a root of unity, our approach can be viewed as evidence for the conjectures in \cite{FH} on two-parameter deformations. 

Let $\Ud_\ell$ be a modified quantum group at an $\ell$-th root of unity. Recall that Lusztig has defined for such algebras a quantum Frobenius map $Fr\colon \Ud_\ell \to \Ud^*_\ell$ whose target is a modified quantum group $\Ud^*_\ell$, where the parameters $v_i$ for $\Ud^*_\ell$ are $\pm 1$ (see Section \ref{background} for details). Under mild restrictions, (which are satisfied for example by all finite and affine types except the affine type $A^{(1)}_{2n}$, which is in any case simply-laced, and so not interesting from the point of view of our duality) if $\Ud_\ell$ is the quantum group attached to $(X,Y,I)$ at $\ell = r$, the algebra $\Ud^*_\ell$ is very close to the enveloping algebra of $^L\Lieg$.  In particular, there is a natural category of representations for $\Ud_\ell^*$ which can be identified with the category $\Oint$ of integrable representations in category $\mathcal O$ for $^L\Lieg$, and moreover the natural notion of characters for representations of each algebra coincide under this identification. 

In \cite{M}, it was observed that the map $Fr$ has a natural splitting $c \colon \Ud^*_\ell \to \Ud_\ell$ (this is readily deduced from the existence of a slightly weaker kind of splitting map which had already been constructed by Lusztig). Now $\Ud_\ell$ has a standard representation $\nabla(\lambda)$ with character $\chi(\lambda)$, and via the map $c$, the subspace $c^*(\nabla(\lambda))$ becomes a representation of $\Ud^*_\ell$ which when viewed as a $^L\mathfrak g$-module lies in $\mathcal O_\text{int}$. It is the easy to check that this gives a representation-theoretic lift of equation (\ref{characterdualityformula}), and hence positivity for the integers $m_\mu^\lambda$. Some examples of this positivity were already established in \cite{FH} where the case of $B_2$ is studied in detail.

Another natural question posed by \cite{FH} was to compute the ``Langlands branching multiplicities'' $m_\mu^{\lambda}$. Somewhat surprisingly, we are able to give an expression for the $m_\mu^\lambda$ in terms of tensor product multiplicities for $\Lieg$ in the case of an arbitrary symmetrizable Kac-Moody Lie algebra. Since tensor product  multiplicities can be calculated via various combinatorial techniques, such as Littelmann paths \cite{Li95}, Kashiwara's crystal graphs (see \textit{e.g.} \cite{K02}), and (in the finite-type case) Berenstein-Zelevinsky's polyhedral expressions \cite{BZ}, this gives a computable expression for the Langlands branching multiplicities. Moreover, since tensor product multiplicities are manifestly positive, we get a purely combinatorial proof of the positivity of branching multiplicities, and hence obtain a proof that $c^*(\nabla(\lambda))$ has the structure of a $^L\Lieg$-representation in the general case (without explicitly constructing that action).

Using the theory of tilting modules and good filtrations of modules for quantum groups at a root of unity, we can also give representation-theoretic meaning to the combinatorial calculation of branching multiplicities. Since the theory needed for this is only available in the finite-type case, this interpretation is limited to that context. It is perhaps interesting to note that our results give a new application of the theory of quantum groups to a question which involves only the classical representation theory of Kac-Moody algebras.

We now briefly outline the organization of the paper: In Section \ref{background} we review Lusztig's modified quantum groups. In Section \ref{contract} we recall the contraction map of \cite{M}, and in Section \ref{duality} use it to construct a Langlands duality for representations in category $\Oint$ for any symmetrizable Kac-Moody algebra satisfying a mild technical condition. In Section \ref{combinatorialbranching} we establish the interpretation of Langlands duality branching multiplicities as tensor product multiplicities in the general case. In Section \ref{branchingtilting} we interpret this combinatorics via tilting modules and good filtrations for quantum groups of finite type.

\section{Quantum Groups and Modified Quantum Groups}
\label{background}
In this section we review Lusztig's modified quantum groups. A detailed account of these algebras is contained in \cite[Part IV]{L93}. Let $\mathbb Q(v)$ be the field of rational functions in an indeterminate $v$ with $\mathbb Q$-coefficients, and let $\mathcal A = \mathbb Z[v,v^{-1}]$, a subring of $\mathbb Q(v)$. Suppose that $C=(a_{ij})_{i,j\in I}$ is a symmetrizable generalized Cartan matrix, and let $(d_i)_{i \in I} \in \mathbb N^I$ be a vector of positive integers such that
\begin{equation}
\label{symmetrization}
d_ia_{ij} = d_ja_{ji}, \quad\forall i,j \in I.
\end{equation}
We may define a symmetric bilinear pairing $x,y \mapsto x \cdot y$ on $\mathbb Z[I]$ by setting
\[
i \cdot j = d_i a_{ij}, \quad i,j \in I.
\]
and extending linearly. The pair $(\mathbb Z[I], \cdot)$ is then a Cartan datum in the sense of Lusztig \cite{L93}
\begin{remark}
When the matrix $C$ is indecomposable (in the sense of Kac) all symmetrizing vectors $(d_i)_{i \in I}$ are multiples of the minimal such vector $(r_i)_{i \in I}$. Set $r = \text{l.c.m.}\{r_i: i \in I\}$ to be the least common multiple of the $r_i$, and similarly $d = \text{l.c.m.}\{d_i: i \in I\}$. All our results, however, make sense in the context of a general Cartan datum.
\end{remark}

The Weyl group attached to a generalized Cartan matrix is the group $W$ generated by involutions $\{s_i: i \in I\}$ satisfying braid relations of length $m_{ij} = 2,3,4,6$ according as $a
_{ij}a_{ji}$ is $0,1,2,3$ (when $a_{ij}a_{ji}$ is at least $4$ no relation is imposed).

For each $i \in I$ let $v_i = v^{d_i}$ and set 
\[
[n]_i = (v_i^n - v_i^{-n})/(v_i-v_i^{-1}) \in \mathcal A.
\] 
We then define 
\[
[n]_i! = [n]_i[n-1]_i\ldots[1]_i; \qquad {n \brack k}_i = \frac{[n]_i!}{[k]_i! [n-k]_i!}.
\]
It is easy to check that these quantum binomial coefficients all lie in $\mathcal A$. 

Given a Cartan datum, we define $\mathbf f$ to be the $\mathbb Q(v)$-algebra generated by symbols $\{\theta_i: i \in I\}$ subject to relations:
\[
\sum_{r+s=1-a_{ij}} \theta_i^{(r)}\theta_j\theta_i^{(s)} = 0,
\]
where $\theta_i^{(n)} = \theta_i^n/[n]_{i}!$ is a quantum divided power. We also need an integral form of $\mathbf f$: let $\fA$ be the $\mathcal A$-subalgebra of $\mathbf f$ generated by $\{\theta_i^{(n)}: n \geq 0, i \in I\}$. Lusztig \cite{L93} has shown that $\fA$ is a free $\mathcal A$-module with a canonical basis $\mathbf B$. Both $\mathbf f$ and $\fA$ are clearly $\mathbb Z[I]$-graded. 

To define the full quantum group we need slightly more data. Fix a Cartan datum $(I, \cdot)$. 
\begin{definition}
\label{rootdatumdef}
Let $C=(a_{ij})_{i,j \in I}$ be a generalized Cartan matrix. A root datum associated to $C$ consists of a pair $(X,Y)$ of finitely generated free abelian groups, a perfect pairing $\langle \cdot, \cdot \rangle \colon Y \times X \to \mathbb Z$, and finite sets $\{\alpha_i\}_{i \in I} \subset X$ and $\{\check{\alpha}_i\}_{i \in I} \subset Y$ consisting of simple roots and coroots respectively. These must satisfy 
\[
\langle \check{\alpha}_i, \alpha_j \rangle = a_{ij}, \quad \forall i, j \in I.
\]
Where there is no possibility for confusion, we normally abuse notation slightly and write $(X,Y,I)$ to denote a root datum. The Weyl group $W$ attached to $C$ acts on $X$ via  
\[
s_i \mapsto (\lambda \mapsto \lambda - \langle \check{\alpha}_i,\lambda\rangle \alpha_i) \in \text{Aut}(X).
\]
and by duality also on $Y$ so that $\langle w(\mu),\lambda\rangle = \langle \mu, w^{-1}(\lambda)\rangle$, for all $w \in W, \mu \in Y, \lambda \in X$. If $(I,\cdot)$ is a Cartan datum, we say that $(X,Y,I)$ is a root datum of type\footnote{In \cite{L93} a root datum is defined in terms of a Cartan datum, and the generalized Cartan matrix is not emphasized.} $(I,\cdot)$ if $(X,Y,I)$ and $(I,\cdot)$ are associated to the same generalized Cartan matrix $C$.
\end{definition}

The simple roots and simple coroots give natural maps from  $\mathbb Z[I]$ to $X$ and $Y$ respectively (which we suppress any notation for).  We say that a root datum is $X$-regular if the simple roots are linearly independent, and $Y$-regular if the simple coroots are. For a finite-type generalized Cartan matrix, $X$ and $Y$ regularity are automatic by nondegeneracy, but in general it is an additional assumption. For a $Y$-regular root datum we set
\[
X^+ = \{\lambda \in X: \langle \check{\alpha}_i, \lambda \rangle \geq 0\}.
\]
For an $X$-regular root datum, we may define a partial ordering on the weight lattice $X$ by setting $\lambda \leq \mu$ if $\mu- \lambda \in \sum_{i \in I} \mathbb N \alpha_i$. Unless otherwise stated, we always assume that our datum is $X$ and $Y$ regular, so that $X$ has both the dominant cone $X^+$ and the partial order.

\begin{definition}
Given a root datum $(X,Y)$ of type $(I, \cdot)$ we may define an associated quantum group $\mathbf U$, which is a $\mathbb Q(v)$-algebra generated by symbols $E_i, F_i, K_\mu$, $i \in I$, $\mu \in Y$, subject to the following relations.
\begin{enumerate}
\item $K_0=1$, $K_{\mu_1}K_{\mu_2} = K_{\mu_1+\mu_2}$ for $\mu_1,\mu_2 \in Y$;
\item $K_{\mu} E_i K_{\mu}^{-1} = v^{\langle\mu,\alpha_i\rangle}E_i, \quad K_{\mu} F_i K_{\mu}^{-1} =
 v^{-\langle\mu,\alpha_i \rangle}F_i$ for all $i \in I$, $\mu \in Y$;
\item $E_iF_j - F_jE_i = \delta_{i,j}\frac{\tilde{K}_i-\tilde{K}_i^{-1}}{v_i-v_i^{-1}}$;
\item The maps $+ \colon \{\theta_i: i \in I\} \to \U$ given by $\theta_i \mapsto E_i$ and $-\colon \{\theta_i \in I\} \to \mathbf U$ given by $\theta_i \mapsto F_i$ extend to homomorphisms $\pm\colon \mathbf f \to \U$.
\end{enumerate}
Here $\tilde{K}_i$ denotes $K_{(i\cdot i/2)\check{\alpha}_i}$. The images of $\mathbf f$ under the maps $\pm$ are denoted $\mathbf U^\pm$. 

\end{definition}

\begin{remark}
Suppose the generalized Cartan matrix is indecomposable, and $(r_i)_{i \in I}$ is the minimal symmetrizing vector. If $(d_i) = d(r_i)$ is another choice of symmetrizing vector, then the quantum group $\mathbf U$ obtained from the datum with $i\cdot j = d_ic_{ij}$ is obtained from the one corresponding to the minimal symmetrization by adjoining a $d$-th root of $v$. Frenkel and Hernandez study two-parameter quantum groups of, for example, types ``$B_1$'' and ``$C_1$'' where the two parameter seem to correspond to different choices of symmetrizations, but the author does not know if this is a useful perspective on their deformations.
\end{remark}

We now recall various categories of modules for $\mathbf U$. Given a left $\mathbf U$-module $V$, and $\lambda \in X$, the $\lambda$-weight space of $V$ is
\[
V_\lambda = \{u \in V: K_\mu.u = v^{\langle \mu, \lambda \rangle}.u, \quad \forall \mu \in Y\}.
\]
A module $V$ is said to be a \textit{weight module} if it is the direct sum of its weight spaces. The full subcategory of the category of $\mathbf U$-modules whose objects are weight modules will be denoted $\text{Mod}_X$, and the full subcategory whose objects are weight modules with finite dimensional weight spaces will be denoted $\text{Mod}_X^f$. In fact we will focus on a smaller full subcategory in $\text{Mod}_X$ consisting of integrable modules with certain bounds on weight spaces. A module $V$ in $\text{Mod}_X^f$ is \textit{integrable} if the actions of $E_i^{(n)}$ and $F_i^{(n)}$ are locally nilpotent for every $i \in I$, $n \in \mathbb N$. For $\lambda \in X$ let $D(\mu) = \{\lambda \in X: \lambda \leq \mu\}$. The module $V$ lies in $\mathcal O_{\text{int}}$ if it is integrable, and there is a finite set of weights $E \subset X$ such that if $V_\lambda \neq \{0\}$ then there some $\mu \in E$ with $\lambda \leq \mu$.

For modules in $\text{Mod}_X^f$ it is possible to define a notion of character. Let $\mathcal E'$ be the abelian group of formal sums $\sum_{\lambda \in X} a_\lambda e^{\lambda}$, where $a_\lambda \in \mathbb Z$. For $\xi \in \mathcal E'$ let $\text{supp}(\xi) = \{ \lambda \in X: a_\lambda  \neq 0\}$, and let $\mathcal E$ be the subgroup of $\mathcal E'$ consisting of those $\xi$ for which $\text{supp}(\xi)$ lies in a finite union of $D(\mu)$s. Given a module $V$ in $\text{Mod}_X^f$, we set
\[
\chi(V) = \sum_{\lambda \in X} \dim(V_\lambda) e^{\lambda} \in \mathcal E'.
\]
Clearly if $V$ lies in $\mathcal O_{\text{int}}$ then $\chi(V) \in \mathcal E$, and this yields an embedding of $K_0(\Oint)$ into $\mathcal E$. It can be shown that the character of a integrable module is invariant under the action of $W$, so the image of $\chi$ in fact lies in $\mathcal E^W$. Notice that although multiplication for elements of $\mathcal E'$ does not necessarily make sense, it is easy to see elements of $\mathcal E$ can be multiplied in the obvious way, so that $\mathcal E$ forms an algebra. 

We now recall the modified version of the quantum group due to Lusztig which is better suited to our purposes. Let $\widehat{\mathbf U}$ be the endomorphism ring of the forgetful functor from $\text{Mod}_X$ to the category of vector spaces. Thus by definition an element of $a$ of $\widehat{\mathbf U}$ associates to each object $V$ of $\text{Mod}_X$ a linear map $a_V$, such that $a_W\circ f = f \circ a_V$ for any morphism $f\colon V \to W$. Any element of $\mathbf U$ clearly determines an element of $\widehat{\mathbf U}$, giving a natural inclusion $\U \hookrightarrow \widehat{\U}$, so one may think of $\widehat{\mathbf U}$ as a sort of completion of $\mathbf U$. For each $\lambda \in X$, let $1_\lambda \in \widehat{\U}$ be the projection to the $\lambda$ weight space. Then $\widehat{\U}$ is isomorphic to the direct product $\prod_{\lambda \in X}\mathbf U 1_\lambda$, and we set $\Ud$ to be the subring (in fact, clearly, $\mathbb Q(v)$-subalgebra)
$$\dot{\mathbf U} = \bigoplus_{\lambda \in X} \mathbf U 1_\lambda.$$ Note that $\Ud$ does not have a multiplicative identity, but instead a collection $\{1_\lambda: \lambda \in X\}$ of orthogonal idempotents. 

It is clear that the category $\text{Mod}_X$ is equivalent to a category of modules for $\Ud$, the category of \textit{unital} modules. A module $V$ for $\Ud$ is said to be unital if for every $v \in V$, there is a finite set $K \subset X$ such that $v = \sum_{\lambda \in K} 1_\lambda(v)$. Let $\text{Mod}_{1}$ be the  category of unital modules for $\Ud$. It is easy to see that $\text{Mod}_X$ is equivalent to $\text{Mod}_1$ (see \cite[23.1.4]{L93}). We denote the full subcategories of $\text{Mod}_1$ corresponding to $\text{Mod}_X^f$, $\mathcal O_{\text{int}}$ by $\text{Mod}_1^f$ and $\dot{\mathcal O}_{\text{int}}$ respectively. The weight spaces of a module in $\text{Mod}_X$ correspond to the images of the operators $1_\lambda$ under this equivalence, hence viewing a module $V$ in $\text{Mod}_X$ as a module for $\Ud$ we may define the character by
\[
\chi(V) = \sum_{\lambda \in X} \dim(\text{im}(1_\lambda)) e^\lambda,
\] 

\begin{definition}
\label{present}
Let $\Uda$ be the $\mathcal A$-subalgebra of $\Ud$ generated by elements $E_i^{(n)}1_\lambda$ and $F_i^{(n)}1_\lambda$ for $\lambda \in X, i \in I$,  and $n \in \mathbb Z_{\geq 0}$. Lusztig \cite[Chapter 25]{L93} has shown that $_\mathcal A \Ud$ is a free $\mathcal A$-module equipped with a canonical basis $\dot{\mathbf B}$. 
\end{definition}

\begin{remark}
\label{Udotinfo}
The canonical basis $\dot{\mathbf B}$ of $\Ud$ was constructed by Lusztig. When the root datum is $X$-regular, in the same way that the canonical basis $\mathbf B$ of $\mathbf f$ yields natural bases for irreducible integrable highest weight modules, $\dot{\mathbf B}$ yields natural bases  for the tensor product of an irreducible integrable highest and lowest weight modules. The basis however can be constructed for an arbitrary root datum.
\end{remark}

\begin{remark}
It is straight-forward to give a presentation for $_\mathcal A\Ud$ via the $\mathbf U^{\pm}$-bimodule structure, so it can be described just as explicitly as $\mathbf U$. Here we have defined it in terms of the category $\text{Mod}_X$ in order to describe the relation between the representations of the two algebras.
\end{remark}

\begin{definition}
Let $(X,Y,I)$ and $(X^\sharp, Y^\sharp, J)$ be root data. Following Steinberg\footnote{And possibly many other. In \cite{L93} Lusztig defines morphisms of root data, but these are stricter than the notion of an isogeny.}, \cite{St} an \textit{isogeny} of root data is a map $\varphi\colon X \to X^\sharp$ and a bijection $ i \leftrightarrow i'$ between $I$ and $J$ such that 
\begin{itemize}
\item 
Both $\varphi$ and its transpose $\varphi^t \colon Y^\sharp \to Y$ are injective.
\item
$\varphi(\alpha_i) = \ell_i \alpha_{i'}$ for some $\ell_i \in \mathbb Z$, and similarly $\varphi^t(\check{\alpha}_{i'}) = \ell_i \check{\alpha}_{i}$.
\end{itemize}
\end{definition}

Given a positive integer $\ell$ and a root datum $(X,Y,I)$, Lusztig has defined an $\ell$-modified root datum\footnote{This is the author's term, but there does not seem to be any established terminology for it.} \cite[2.2.4]{L93}:

\begin{definition} Let $(I,\cdot)$ be a Cartan datum, and let $\ell$ be a positive integer. The $\ell$-modified Cartan datum $(I,\circ)$ is given by
\[
i \circ j = l_il_j(i\cdot j),
\]
where $l_i$ is the smallest positive integer such that $l_i(i\cdot i/2) \in \ell \mathbb Z$. It is easy to check that $(I, \circ)$ is indeed a Cartan datum. Given a root datum $(X,Y,I)$ of type $(I,\cdot)$ we can also define a new root datum of type $(I,\circ)$ by setting $X^*= \{\lambda \in X\colon \langle \check{\alpha}_i, \lambda \rangle \in l_i \mathbb Z\}$ and $Y^* = \text{Hom}(X^*, \mathbb Z)$, with the obvious pairing between $X^*$ and $Y^*$. The simple roots of the $(X^*,Y^*)$ are $\alpha_i^* = l_i \alpha_i$ and the simple coroots are $\check{\alpha}_i^*$, where $\check{\alpha}_i^*(\lambda) = l_i^{-1}\langle \check{\alpha}_i, \lambda \rangle$. Note that the inclusion $X^* \to X$ and its transpose, the induced restriction map $Y \to Y^*$ yield an isogeny from $(X,Y,I)$ to $(X^*,Y^*,I)$. It is easy to see that if the simple roots and coroots are linearly independent in $(X,Y,I)$ then the same is true for $(X^*,Y^*,I)$. 
\end{definition}

\begin{remark}
\label{Weylremark}
Note that it is immediate from the definitions that the Weyl group $W^*$ of $(I,\circ)$ is canonically isomorphic to $W$ the Weyl group of $(I,\cdot)$. Moreover, by checking on the generators $s_i$, it is easy to see that the action of $W^*$ on $X^*$ coincides with the restriction of the action of $W$ on $X$ via this isomorphism. We may thus identify the two groups and write $W$ for the Weyl group in either case\footnote{In \cite{FH} the fact that the Weyl group actions coincide is Lemma $2.2$. The root datum formalism however makes this check self-evident.}. In particular note that this shows $X^*$ is invariant under the action of $W$.
\end{remark}

\begin{remark}
\label{Cartanremark}
If $C$ is the generalized Cartan matrix of $(X,Y,I)$ then the generalized Cartan matrix of the $\ell$-modified datum is $LC$ where $L = (\ell_i^{-1}\ell_j)$. In the case where $d$ divides $\ell$, so that $\ell_i = \ell/d_i$ it follows that $LC = C^t$, that is, the $\ell$-modified root datum is attached to the transpose of $C$, and hence the quantum group associated to it is Langlands dual to that attached to the original datum. Notice also that, given $d$ divides $\ell$, the lattice $X^* \subset X$ depends only on the ratio $\ell/d$. More precisely, if the symmetrization vector $(d_i)_{i \in I}$ determining the Cartan datum is a multiple of another symmetrization vector $(c_i)_{i \in I}$, say $d_i = mc_i$, and we write $X^*_{\ell,d}$ for the sublattice given by $\ell$ and $(d_i)_{i \in I}$ then $X^*_{\ell,d} = (\ell/d) X^*_{c,c}$ where $c$ is the least common multiple of the $c_i$s.
\end{remark}

Let $Q = \bigoplus_{i \in I} \mathbb Z \alpha_i$ be the root lattice of $(X,Y,I)$, and $Q^*$ the root lattice of $(X^*, Y^*,I)$.  Let $\Phi$ denote the roots of $(X,Y,I)$ and $\Phi^*$ denote the roots of $(X^*,Y^*,I)$.

\begin{lemma}
\label{scaling}
Suppose that $(X,Y,I)$ is of finite or affine type. Then the inclusion $\iota \colon X^* \to X$ induces a bijection $\alpha \leftrightarrow \alpha^*$ between $\Phi$ and $\Phi^*$ such that $\alpha^* = l_\alpha \alpha$, for some $l_\alpha \in \mathbb Z$.
\end{lemma}
\begin{proof}
For a simple root $\alpha_i$ the result follows from the definition. Next, as noted above, that the Weyl groups for the two root systems are identical, and hence if $\alpha \in \Phi$ is a real root, we may write it as $w(\alpha_i)$ for some $w \in W$ and $i \in I$, and thus 
\[
\alpha^* =w(\alpha_i^*) = l_iw(\alpha_i) = l_i \alpha
\]
(and hence $l_{\alpha} = l_i$). It remains to consider the imaginary roots, in the affine case. Let $(\cdot, \cdot)$ denotes an invariant symmetric bilinear form on $\mathbb Q \otimes_{\mathbb Z} Q$. Then the imaginary roots are exactly the elements $\alpha$ of $Q$ with $(\alpha,\alpha)=0$, and similarly the imaginary roots of $\Phi^*$ are exactly the elements of $Q^*$ of norm zero. Since $Q^* \subset Q$ and the space of norm zero elements is one dimensional, the result follows immediately.
\end{proof}

\begin{remark}
In fact, using \cite[Remark 6.1]{Kac} it is easy to see that $\delta^* = \ell \delta$, where $\delta$ and $\delta^*$ are the smallest imaginary roots of $\Phi$ and $\Phi^*$ respectively. It seems reasonable to conjecture that this Lemma holds for all symmetrizable root data.
\end{remark}

\begin{remark}
\label{finitetypecase}
When $C$ is an indecomposable Cartan matrix (\textit{i.e.} of finite type), the numbers $r_i=\frac{1}{2}(i \cdot i)$ are either $1$ or $r$ in the notation of \cite{FH}. Thus taking any symmetrizing vector $(d_i)_{i \in I}$ and $\ell$ divisible by $d$ we have
\[
l_i = d/d_i = r/r_i = r-r_i +1.
\]
Finally, the lattice here denoted $X$ is the weight lattice denoted $P$ in \cite[2.1]{FH} and the sublattice $X^*$ is there denoted $P'$ (thus we will define the weight lattice of our Langlands dual algebra to be $X^*$, rather than identifying it with $X^*$ as in \cite{FH}). It is interesting to note also that in the affine case one again has $r_i=r$ except in the case of $A_{2\ell}^{(2)}$ where $r=4$, and the $r_i$ take the values $1$,$2$ and $4$. Choosing $\ell =2$ in this case gives an isogeny to a root datum which is not of finite or affine type. While $A_{2\ell}^{(2)}$ is self-dual, and so excluded from the considerations of \cite{FH2}, our constructions should still give some interesting relations between representations of this algebra.
\end{remark}

\subsection{Roots of unity}

Let $\ell$ be any positive integer. Following \cite[35.1.3]{L93}, for $\ell$ even we set  $l=2\ell$, while if $\ell$ is odd we set $l= \ell$ or $2\ell$. We set $\A_\ell$ to be the quotient ring $\A/(\Phi_{l}(v))$ where $\Phi_l$ is the $l$-th cyclotomic polynomial, and then let $\Ud_{\ell}$ be the corresponding specialization $\A_\ell \otimes_{\A} (\Uda)$ of the modified form $\Ud$, and similarly let $\mathbf f_\ell$, $\mathbf U^{\pm}_\ell$ be the specializations $\A_\ell \otimes (\fA)$, $\A_\ell\otimes (_A\mathbf U^{\pm})$. More generally for any $\A$-algebra $R$ we will write $_R\Ud$ for the corresponding specialization of $\Uda$, and $_R\mathbf f$ for the specialization of $\fA$.

Let $\Ud^*$ be the modified quantum group attached to the $\ell$-modified root datum $(X^*,Y^*,I)$. To distinguish it from $\Ud$, we will write its generators as
\[
e_i^{(n)}1_\lambda, f_i^{(n)}1_\lambda, \quad (n \geq 0, i \in I, \lambda \in X^*),
\]
and write the generators of $\mathbf f_\ell$ as $\{\vartheta_i: i\in I\}$.
Since in $\mathcal A_\ell$ we have $(v_i^*)^2 = (v_i^{l_i^2})^2 = 1$, so that $v_i^* = \pm 1$, the algebra $\Ud^*_\ell$ is close to the classical enveloping algebra. 
In this note, we are interested in the most degenerate case, when $\ell = r$.

\section{The contracting homomorphism}
\label{contract}

In this section we recall the contracting homomorphism of \cite{M}, which gives an embedding of the algebra $\Ud^*_\ell$ into $\Ud_\ell$. This relies on the work of Lusztig on the quantum Frobenius homomorphism. Recall from \cite[chapter 35]{L93} that (under mild restrictions on $\ell$ -- see the remark below) there are two $\A_\ell$-homomorphisms $Fr\colon \mathbf f_\ell \to \mathbf f^*_\ell$ and $Fr' \colon \mathbf f^*_\ell \to \mathbf f_\ell$ which are given on generators by $Fr'(\vartheta_i^{(n)}) = \theta_i^{(nl_i)}$, and 
\[
Fr(\theta_i^{(n)}) = \left\{\begin{array}{cc}\vartheta_i^{(n/l_i)}, & \text{ if } l_i | n, \\0, & \text{otherwise}.\end{array}\right.
\]
Lusztig also shows that $Fr$ ``extends'' to a map $Fr\colon \Ud_\ell \to \Ud^*_\ell$ (in the sense that $\Ud_\ell$ is naturally a bimodule over $\mathbf U^\pm_\ell$, and $Fr$ on $\Ud_\ell$ is compatible with these bimodule structures). It is characterized by the conditions:
\[
Fr(E_i^{(n)}1_\lambda) = \left\{\begin{array}{cc} e_i^{(n/l_i)}1_\lambda, & \text{ if } l_i | n, \lambda \in X^*\\0, & \text{otherwise}.\end{array}\right.
\]
 and similarly $Fr(F_i^{(n)}1_\lambda) = f_i^{(n/l_i)}1_\lambda$, if $l_i$ divides $n$ and $\lambda \in X^*$, and to zero otherwise. Note that $Fr$ is obviously surjective.  A simple observation of \cite{M} is that the map $Fr'$ also has an extension to $\Ud^*_\ell$. Here the use of the modified form is essential, as although $Fr$ on $\mathbf U^{\pm}_\ell$ does extend to the ordinary quantum group $\mathbf U_\ell$, the map $Fr'$ does not extend to $\mathbf U^*_\ell$. The existence of the contraction map and the quantum Frobenius are (currently) conditional on some mild technical hypotheses. 

\begin{definition}
\label{oddcycles}
Let $C= (a_{ij})_{i,j \in I}$ be a generalized Cartan matrix. An odd cycle in $C$ is a sequence $i_1,i_2,\ldots,i_{p+1} = i_1$ in $I$ such that $p \geq 3$ is odd, and $a_{i_si_{s+1}} <0$ for each $s=1,2,\ldots,p$, that is, a cycle of odd length in the associated Coxeter graph. A generalized Cartan matrix $C$ has no odd cycles if and only if there is a function $i \mapsto a_i \in \{0,1\}$ such that $a_i+a_j=1$ whenever $a_{ij}<0$ (since a graph with no odd cycles is bipartite).
\end{definition}

A Cartan datum or root datum is said to have no odd cycles if its associated generalized Cartan matrix has no odd cycles. Note that this condition is satisfied by all finite-type and affine Cartan data except $A^{(1)}_{2n}$, and in that case the datum is self-dual, so it will not be of interest to us. 

\begin{prop}\cite{M}
\label{contracthom}
Suppose that $(X,Y,I)$ is a root datum, and let $\phi\colon \A \to R$ be a homomorphism which factors through the natural map $\A \to {\A_\ell}$ for some positive integer $\ell$. If $\ell$ is even assume that $(X,Y,I)$ has no odd cycles. Then there is a homomorphism $c\colon _R\Ud^* \to$ $_R\Ud$ given on generators by $e_i^{(n)}1_\lambda \mapsto E_i^{(nl_i)} 1_\lambda$ and $f_i^{(n)}1_\lambda \mapsto F_i^{(nl_i)} 1_\lambda$ where $\lambda \in X^* \subset X$.
\end{prop}

\begin{remark}
\label{contractionremark}
The proof of the existence $Fr$ in  \cite[chapter 35]{L93} holds with mild restrictions on $\ell$ in addition to the condition of no odd cycles, which in fact are not valid when $\ell = d$. For finite-type quantum groups Kaneda \cite{Ka} has verified that these restrictions can be removed, and indeed this fact was already stated in \cite{L93}. The existence of the map $c$ however, depends only on the existence of the map $Fr'$, not on $Fr$, and thus it is known to exist whenever the root datum has no odd cycles for arbitrary $\ell$, and for an arbitrary root datum if $\ell$ is odd. Indeed minor modifications of the arguments in \cite{L93} allow you to verify the existence of $Fr'$ without the restriction of no odd cycles when $\ell$ is odd, as is briefly sketched in \cite[35.5.2]{L93}.

Finally, it is worth noting that, as Lusztig already points out in \cite{L93}, the quantum Frobenius at $\ell =r$ provides a quantum analogue of Chevalley's exceptional isogenies between algebraic groups in positive characteristic. For example, in characteristic $2$, there is a natural map $\text{SO}_{2n+1} \to \text{Sp}_{2n}$ induced by the quotient map $\mathsf k^{2n+1} \to \mathsf k^{2n+1}/L$ where $L$ is the line fixed by $\text{SO}_{2n+1}$ (\textit{e.g.} the line spanned by $e_0$ if the quadratic form is $x_0^2 +\sum_{i=1}^n x_ix_{i+n}$). The quantum Frobenius for $B_n$ at $\ell=2$ is a lifting of this isogeny to characteristic zero. An elegant construction of all possible isogenies between reductive algebraic groups in any characteristic is given in \cite{St}.
\end{remark}

\section{Duality for representations}
\label{duality}
Let $\mathfrak g$ be a symmetrizable Kac-Moody Lie algebra with indecomposable generalized Cartan matrix $C$, and let $\mathcal O_{\text{int}}(\mathfrak g)$ be the full subcategory of the category of $\mathfrak g$-representations consisting of those representations $V$ which satisfy
\begin{enumerate}
\item
$V$ is a direct sum of its weight spaces $V_\lambda$, where $\lambda \in X$ the weight lattice.
\item 
The operators $e_i$, $f_i$ act locally nilpotently.
\item
There is a finite set of weights $K \subset X$ such that whenever the weight space $V_\mu \neq 0$ there is a $\lambda \in K$ with $\mu \leq \lambda$.
\end{enumerate}
By results of Gabber and Kac \cite{Kac}, the modules in category $\Oint(\mathfrak g)$ are completely reducible, and the simple modules are the standard modules $\{\nabla(\lambda): \lambda \in X^+\}$, whose characters are given by the Weyl-Kac character formula. Let $^L\mathfrak g$ be the Langlands dual Lie algebra, with generalized Cartan matrix $C^t$. In this section we will establish a duality between representations in the categories $\Oint(\mathfrak g)$ and $\Oint({^L\mathfrak g})$, provided $C$ has no odd cycles (see Definition \ref{oddcycles}).

Picking a symmetrizing vector $(d_i)_{i \in I}$ for $C$ we obtain a Cartan datum, and let $(X,Y,I)$ be an associated root datum which is $X$ and $Y$ regular (c.f. the paragraph after Definition \ref{rootdatumdef}). To establish the duality we will use the contraction map of the previous section. For the rest of this section unless otherwise stated, we assume that $d$ divides $\ell$ so that the generalized Cartan matrix of the $\ell$-modified datum is the transpose of $C$ (see Remark \ref{Cartanremark}).

Let $\Ud$ be the modified quantum group associated to the root datum $(X,Y,I)$. The category $\dot{\mathcal O}_\text{int}$ for $\Ud$ which was defined in Section \ref{background} gives a natural deformation of the $\mathcal O_{\text{int}}(\mathfrak g)$. Indeed over $\mathbb Q(v)$ (which we shall refer to as the ``generic''  case) it is a semisimple category (this follows from \cite[Chapter 6]{L93}.
Moreover, for each $\lambda \in X^+$ one can define highest weight modules $\nabla(\lambda)$ over $\mathcal A$ which are integral forms of the simple modules in $\dot{\mathcal O}_{\text{int}}$. Let $\Ud_1$ denote the algebra $\Uda$ specialized at $v =1$. It is shown in \cite[33.1.2]{L93} that the structure of an integrable highest weight module for $\mathbb C \otimes_\mathbb Z\Ud_1$ is equivalent to the structure of a $\Lieg$-integrable highest weight module, in a fashion which preserves weights and hence characters. It follows that the characters of the modules $\nabla(\lambda)$  are given by the Weyl-Kac character formula, so that $\dot{\mathcal O}_{\text{int}}$ gives a deformation of category $\Oint(\mathfrak g)$ in a fashion which preserves characters.

For each $\lambda \in X^+$ we may specialize the module $\nabla(\lambda)$ to $\mathcal A_\ell$ and obtain a so-called standard module which we will also write as $\nabla(\lambda)$ since the context should prevent any possibility of confusion.  Clearly the standard modules have characters given by the Weyl-Kac formula (for the root datum $(X,Y,I)$).

Recall that in $\Ud^*_\ell$ the parameters $v_i^* = \pm 1$. In \cite[33.2]{L93} specializations with this property are called \textit{quasiclassical}. Under the assumption that the root datum has no odd cycles, these specializations are in fact isomorphic to $\Ud^*_1$. More precisely, let $\phi\colon \mathcal A \to R$ be an $\mathcal A$-algebra, such that $\phi(v_i) \in \{\pm 1\}$ for each $i$. Then let $R_0$ be the same ring $R$ with the $\mathcal A$-algebra structure given by mapping $v \mapsto 1$. 

\begin{prop}
\label{quasiclassical}
Let $(X,Y,I)$ be a root datum with no odd cycles. Then there is an isomorphism of the algebras $_R\Ud$ and $_{R_0}\Ud$. Moreover, the isomorphism maps $1_\zeta \in {_R\Ud}$ to $1_\zeta \in {_{R_0}\Ud}$ ($\zeta \in X$) so that pulling back representations via this isomorphism preserves characters.
\end{prop}
\begin{proof}
The isomorphism was constructed by Lusztig in \cite[33.2.3]{L93}. Looking at the formulas there immediately establishes the assertion about characters.
\end{proof}

Thus $\Ud^*_\ell$ is isomorphic to the algebra $\mathcal A_\ell\otimes_\mathbb Z \Ud^*_1$. Extending scalars to $\mathbb C$ by picking a primitive root of unity, the discussion above (applied to $\Ud^*$ instead of $\Ud$) shows that the category $\dot{\mathcal O}_{\text{int}}$ for $\Ud^*_\ell$ is equivalent to category $\Oint$ for $^L\Lieg$, in a manner which preserves characters. 

We can now describe our duality. It is in the same spirit as, but thanks to the map $c^*$, simpler than, the representation-theoretic duality for two-parameter quantum groups proposed in \cite{FH}. Suppose that $W$ is a representation of $\Lieg$ in category $\Oint$. Then we may equivalently think of it as a representation of $\Ud_1$, and as such it is the specialization of a representation $W_v$ of $\Uda$. We may then specialize this representation instead to $\Ud_\ell$, to obtain a representation $W_\ell$. Using the contraction map $c$ we then obtain a representation of $\Ud^*_\ell$. Since any weight space of $W_\ell$ with weight $\mu \notin X^*$ is annihilated by $\Ud_\ell^*$, we define $c^*(W_\ell)$ to be the subspace of $W_\ell$ which is the direct sum of those weight spaces whose weight lie in $X^*$, taken as a $\Ud_\ell^*$ representation. Then by the above discussion we may view $c^*(W_\ell)$ as a representation of $^L\Lieg$ which is easily seen to lie in $\mathcal O_{\text{int}}({^L\mathfrak g})$. By abuse of notation, we will write $c^*(W)$ instead of $c^*(W_\ell)$ if there is no danger of confusion.

\begin{remark}
Notice that although our construction works on the level of representations, it is not given functorially, since we really only deform simple representations in $\mathcal O_{\text{int}}(\mathfrak g)$. This is perhaps due to the author's ignorance, and it would be interesting to know to what extent this can be refined.

Note also that if $\ell$ is odd, then the above duality can be constructed for an arbitrary root datum: the contraction homomorphism exists in this case, as was noted before in Remark \ref{contractionremark}, and moreover the same modification to Lusztig's proof of the existence of $Fr'$ (that is, using \cite[33.1]{L93} rather than \cite[33.2]{L93}) establishes the necessary relation between $\Ud^*_\ell$ representations and $^L\mathfrak g$ representations. 
\end{remark}

We now wish to study this duality at the level of characters and relate it to the duality of \cite{FH}. Recall the ring $\mathcal E$ introduced in Section \ref{background}. If $V$ is in category $\Oint$ for $\Lieg$ then we set, just as for representations of $\Ud_\ell$,
\[
\chi(V) = \sum_{\lambda \in X} \text{dim}(V_\lambda) e^{\lambda} \in \mathcal E.
\]
and this embeds $K_0(\Oint)$ into $\mathcal E^W$. Let $\chi^L$ be the corresponding map for $\Ud^*_\ell$. It is immediate from the definitions that
\[
\chi^L(c^*(V)) = \Pi_\ell(\chi(V))
\]
where $\Pi_\ell(e^{\lambda}) = e^\lambda$ if $\lambda$ lies in $X^*$ and zero otherwise. For $\lambda \in X^+$ let $\chi(\lambda)$ be the Weyl character attached to $\lambda$, and similarly for $\mu \in X^+ \cap X^*$ let $\chi^L(\mu)$ be the Weyl character (for the datum $(X^*, Y^*,I)$) attached to $\mu$. Note that since $X^*$ is $W$-invariant (see Remark \ref{Weylremark}), the map $\Pi$ is $W$-equivariant for the obvious actions of $W$ on $\mathbb Z[X]$ and $\mathbb Z[X^*]$ respectively.

\begin{remark}
To compare with \cite{FH}, where $C$ is an indecomposable Cartan matrix, and take $\ell=d$ the least common multiple of the $d_i$s. Recall that by Remark \ref{finitetypecase} in this case we have $l_i = r/r_i = 1+r-r_i$. In \cite{FH}, the authors use an auxillary map $P' \to P^L$ (in the notation of that paper), given by
\[
\lambda \mapsto \sum_{i \in I} \lambda(\check{\alpha}_i)(1+r-r_i)^{-1}\check{\varpi}_i, \lambda \in P',
\]
and they define $\Pi \colon P \to P^L$ by extending this map by zero outside $P'$. In the notation of this paper, we have identified $P'$ with $P^L$ as $X^*$, (this is implicit in attaching a full root datum to $X^*$). Via these identifications, the map $\Pi_d$ here coincides with the map $\Pi$ of \cite{FH}. \end{remark}

\begin{prop}
Suppose that $(X,Y,I)$ is a root datum with no odd cycles, or that $\ell$ is odd, and let $\lambda \in X^+\cap X^*$. Then
\[
\Pi_\ell(\chi(\lambda)) = \chi^L(\lambda) + \sum_{\mu < \lambda, \mu \in X^*} m_{\mu}^\lambda \chi^L(\mu).
\]
where $m_\mu^\lambda$ is a nonnegative integer.
\end{prop}
\begin{proof}
We simply apply our representation-theoretic duality to the simple highest weight module $V$ of highest weight $\lambda$. Then $V_\ell$ is the standard module $\nabla(\lambda)$ for $\Ud_\ell$ of highest weight $\lambda$. Taking $c^*(V_\ell)$ we obtain a representation for $\Ud^*_\ell$ which has character $\Pi_\ell(\chi(\lambda))$ by the remark preceeding the proposition. But as discussed above, the representations of $\Ud^*_\ell$ in category $\dot{\mathcal O}_{\text{int}}$ are semisimple and the characters of simples are given by Weyl's formula, thus $\chi^L(c^*(V_\ell))$ is a positive sum of Weyl characters $\chi^L(\mu)$. Since all the weights in $\nabla(\lambda)$ are less than $\lambda$, we must also have $\mu \leq \lambda$, and as the highest weight occurs with multiplicity $1$,we see the coefficient of $\chi^L(\lambda)$ must be $1$ as claimed.
\end{proof}

\begin{remark}
Essentially the same map on representations is considered by Littelmann in \cite{Li} in his construction of standard monomials via quantum groups. One could consider the positive characteristic analogue of $c$ (the Frobenius splitting map), but then the categories of representations one has to consider are more complicated.
\end{remark}

\begin{remark}
In \cite{FH} the authors connect the character duality map to two-parameter deformations of quantum groups, where one of the parameters is specialized to a root of unity. It would be very interesting to understand what those deformations have to do with quantum isogenies. In another direction, the map $c$ acts on representations which are not necessarily in category $\dot{\mathcal O}_{\text{int}}$. For example, one could consider integrable representations of affine quantum groups at level zero. It was shown in \cite{M} that $c$ is compatible with extremal weight modules in level zero, thus it seems likely that it would act sensibly on $q$-characters (which however are usually defined only on finite dimensional representations). Very recent work of Frenkel and Hernandez \cite{FH2} investigates a duality for $q$-characters, and earlier work of Frenkel and Mukhin \cite{FM} has already studied a notion of $q$-characters at roots of unity. One can hope the techniques of this paper can be connected to these theories.

\end{remark}

We end this section with an application of the above results. In \cite{M} it is asserted that $c$ is an embedding. Since this was proved using the Frobenius map $Fr$, which is not known to exist in the same generality as $c$ is, we give an alternative proof of this fact so that the map $c^*$ is always a restriction of representations to a subalgebra. Note that here we need not assume that $\ell$ is divisible by $d$ as this was done earlier only to ensure that $\Ud^*_\ell$-representations correspond to $^L\mathfrak g$-representations. The semisimplicity of representations of $\Ud^*_\ell$ in $\dot{\mathcal O}_{\text{int}}$ is all that is needed in the following. 

\begin{lemma}
The map $c\colon \Ud^*_\ell \to \Ud_\ell$ is injective.
\end{lemma}
\begin{proof}
Note that since $\Ud^*_\ell$ is free as an $\mathcal A_{\ell}$-module, thus it suffices to check the map is an injective after we extend scalars to $\mathbb C$. Let $\omega\colon \Ud^* \to \Ud^*$ be the involutive automorphism  given by $\omega(e_i^{(n)}1_\lambda) = f_i^{(n)}1_{-\lambda}$. Twisting by $\omega$ interchanges highest and lowest weight modules. For $\lambda, \mu \in X^+$ let $\nabla^*(\lambda)^{\omega}$ and $\nabla^*(\mu)$ be the standard modules of lowest weight $-\lambda$ and highest weight $\mu$ respectively. By \cite[23.3.8]{L93} If $v_\lambda \in \nabla^*(\lambda)^{\omega}$ and $v_{\mu}\in \nabla^*(\mu)$ are of weight $-\lambda$ and $\mu$ respectively, and $\zeta = \mu-\lambda$, then the map $\Ud1_\zeta \to\nabla^*(\lambda)^\omega \otimes \nabla^*(\mu)$ given by $u1_\zeta \mapsto u1_\zeta(v_{\lambda}\otimes v_\mu)$ is surjective. Let $P(\zeta,\lambda,\mu)$ be its kernel. It follows readily from the results of \cite[23.3]{L93} that the intersection of these ideal over all $\lambda, \mu$ with $\lambda-\mu = \zeta$ is zero. Note moreover that these results all hold over $\mathcal A$, and hence over any ring (see \cite[31.2]{L93} for details). 

Next we have already seen that for $\lambda \in X^*\cap X^+$ the module $c^*(\nabla(\lambda))$ contains $\nabla^*(\lambda)$ (here we use the fact that we are in the quasiclassical case, so that $c^*(\nabla(\lambda))$ is semisimple and $\nabla^*(\mu)$ is simple for all $\mu \in X^*$), and the highest weight spaces correspond. But now one can check directly that $c$ is compatible with the ``coproduct'' on $\Ud_\ell$ and $\Ud^*_\ell$ (see \cite[23.1.5]{L93}), so that it respects tensor products of modules. Thus if $u1_\zeta \in \Ud^*_\ell$ lies in the kernel of $c$, then $u1_\zeta$ annihilates $c^*(\nabla(\lambda)^\omega \otimes \nabla(\mu))$ for all $\lambda, \mu \in X^*$ with $\lambda - \mu = \zeta$, and hence all the modules $\nabla^*(\lambda)^\omega \otimes \nabla^*(\mu)$, so by the above it must be zero as required.
\end{proof}

\begin{remark}
The existence of special isogenies (the positive characteristic analogues of the $\ell=r$ quantum Frobenius for finite type quantum groups) has been used by Kumar and Stembridge \cite{KS} to establish certain inequalities on tensor product multiplicities for a group and its dual group. They use characteristic $p$ methods rather than quantum groups, but one can also use quantum isogenies to establish their results. It would be interesting to know if the splitting map is useful in their context.
\end{remark}

\section{On Langlands duality branching rules: combinatorics.}
\label{combinatorialbranching}

In this section we show that the multiplicities $m_\mu^\lambda$ which occur in the character duality can be interpreted as tensor product multiplicities. We shall work purely combinatorially, and so $C$ can be an arbitrary symmetrizable generalized Cartan matrix (but we assume still that our root datum is $X$ and $Y$ regular so that we have a dominant cone $X^+$ and partial order $\leq$ on $X$). Since tensor product multiplicities are manifestly positive, we also obtain a proof that $\Pi_\ell(\chi(\lambda))$ is the character of a representation of the dual Lie algebra for any symmetrizable Kac-Moody Lie algebra. 

Let $\Phi$ be the set of roots of the Kac-Moody Lie algebra attached to $(X,Y,I)$ and $\Phi^*\subset X^*$ the set of roots for the dual Lie algebra. Pick a Weyl vector $\rho \in X$, so that $\langle \check{\alpha}_i, \rho\rangle = 1$ for all $i \in I$. If $\{\varpi_i: i \in I\}$ denote a choice of fundamental weights (which we may assume lie in $X$) then we may take $\rho = \sum_{i \in I} \varpi_i$. We write $w\cdot \lambda$ for the $\rho$-shifted action of $W$ on $X$, that is
\[
w\cdot \lambda = w(\lambda + \rho) - \rho.
\]

If $\nu \in X^+$, then for all $w \in W$ we have $\nu - w(\nu) \in \sum_{i \in I}\mathbb N\alpha_i$. It follows that $A_\nu = \sum_{w \in W} \varepsilon(w) e^{w\cdot \nu}$ lies in $\mathcal E$ for any $\nu \in  X^+$. Since $A_{w\cdot \nu} = \varepsilon(w)A_\nu$ we see $A_\nu \in \mathcal E$ for any $\nu \in W\cdot X^+$. Set $\mathfrak D = \prod_{\alpha \in \Phi, \alpha>0} (1-e^{-\alpha})^{n_\alpha} \in \mathcal E$, where $n_\alpha$ is the dimension of the $\alpha$ root space in the Kac-Moody algebra. For $\lambda \in X^+$, the Weyl-Kac character formula for the irreducible highest weight representation $\nabla(\lambda)$ of highest weight $\lambda \in X^+$ states that
\[
\chi(\lambda) = \chi(\nabla(\lambda)) = \mathfrak D^{-1}.A_\lambda,
\]
(note that $\mathfrak D^{-1}$ and $A_\lambda$ lie in $\mathcal E$, so the product is well-defined). Note that, as for the $A_\nu$, the above expression makes sense for any $\lambda \in W\cdot X^+$, not just $\lambda \in X^+$, but clearly if $v \in W$ then $\chi(v \cdot \lambda) = (-1)^{\ell(v)}\chi(\lambda)$, with both sides zero when $\lambda + \rho$ is not regular.

Let $\rho^L = \sum_{i \in I} l_i \varpi_i \in X^*$, a Weyl vector for the datum $(X^*,Y^*,I)$. Then we have a similar expression for the characters of the Langlands dual Kac-Moody algebra in terms of its positive roots and the $\rho^L$-shifted action of $W$ (recall by Remark \ref{Weylremark} that the Weyl groups of the two root data are canonically isomorphic and the natural action of $W$ on $X^*$ is the restriction of that on $X$). We will need some combinatorial lemmas. The key formula in the next lemma goes back to Brauer.

\begin{lemma}
\label{combinatorialavoidance}
Let $\xi \in \mathcal E^W$, that is, $\xi = \sum_{\lambda \in X} a_\lambda e^{\lambda} \in \mathcal E$ and $a_\lambda = a_{w(\lambda)}$ for all $w \in W$, $\lambda \in X$. Then 
\[
\xi.\chi(\nu) = \sum_{\lambda \in X} a_\lambda \chi(\lambda+ \nu).
\]
Moreover it follows that if $\text{supp}(\xi) \subset X\backslash X^*$, that is, $a_\lambda = 0$ if $\lambda \in X^*$, then $\xi.\chi(\rho^L-\rho)$ lies in the span of Weyl characters of the form $\chi(\mu)$ with $\mu \in X^+$ and $\mu \notin (\rho^L-\rho) + X^*$.
\end{lemma}
\begin{proof}
Clearly for the first part of the Lemma it is enough to establish the identity:
\[
\xi.A_\nu = \sum_{\lambda \in X} a_\lambda A_{\lambda+ \nu}.
\]
Now we have 
\[
\begin{split}
\xi.A_\nu &= \sum_{\lambda \in X}\sum_{w \in W} a_\lambda e^{\lambda} \varepsilon(w)e^{w\cdot \nu} \\
&= \sum_{\lambda \in X}\sum_{w \in W} a_\lambda \varepsilon(w) e^{w\cdot(\nu + w^{-1}(\lambda))}.
\end{split}
\]
Interchanging the order of summation we get:
\[
\begin{split}
&= \sum_{w\in W}\sum_{\lambda \in X}a_{\lambda}\varepsilon(w) e^{w\cdot(\nu+w^{-1}(\lambda))} \\
&=  \sum_{w\in W}\sum_{\eta \in X} a_\eta \varepsilon(w)e^{w\cdot(\nu+\eta)}, \quad (\eta = w^{-1}(\lambda)),\\
&= \sum_{\eta \in X} a_\eta A_{\nu+\eta},
\end{split}
\]
where in the second line we used the $W$-invariance of the $a_\lambda$s, and we reversed the order of summation again in the last line. Note that since $\nu+\eta$ may not be dominant, this is not necessarily a positive sum of $A_\lambda$s for $\lambda \in X^+$ even if the $a_\mu$s are all positive integers. 

Applying the formula in the first part of the Lemma to $\xi \in \mathbb Z[X\backslash X^*]^W$ and $\nu = \rho^L-\rho$ and using the fact that $\chi(w\cdot \lambda)= (-1)^{\ell(w)}\chi(\lambda)$, we see that all the Weyl characters which can occur in the product $\xi.\chi(\rho^L-\rho)$ have highest weights of the form 
\[
w(\rho^L+\eta)-\rho = w(\eta) + (w(\rho^L)-\rho^L) +(\rho^L-\rho),
\]
where $\eta \in X^+$ and $a_\eta \neq 0$. But then by assumption $\eta \notin X^*$, and so $w(\eta) \notin X^*$. On the other hand clearly $w(\rho^L) -\rho^L \in X^*$, hence this weight cannot lie in $\rho^L-\rho + X^*$, and we are done.
\end{proof}

\begin{lemma}
\label{combinatorialSt}
Let $\lambda \in X^*$. Then we have 
\[
\chi(\lambda+\rho^L-\rho) = \chi(\rho^L-\rho).\chi^L(\lambda).
\]
\end{lemma}
\begin{proof}
Let $\Phi$ denote the set of roots for $(X,Y,I)$, and similarly let $\Phi^*$ denote the set of roots for $(X^*, Y^*,I)$. We have
\[
\chi(\lambda+\rho^L-\rho) = (\prod_{\alpha \in \Phi, \alpha >0} (1-e^{-\alpha})^{n_\alpha} \sum_{w \in W} \varepsilon(w)e^{w(\lambda+\rho^L)-\rho},
\]
while 
\[
\chi^L(\lambda) = \big(\prod_{\alpha^* \in \Phi^*, \alpha^* >0} (1-e^{-\alpha^*})^{-n_\alpha^*}\big) \sum_{w \in W} \varepsilon(w)e^{w(\lambda+\rho^L)-\rho^L},
\]
where $n_\alpha$ and $n_\alpha^*$ denote the dimensions of the roots spaces in $\mathfrak g$ and $^L\mathfrak g$ respectively. But now we have
\[
\begin{split}
\chi(\rho^L-\rho) &= (\prod_{\alpha >0} (1-e^{-\alpha})^{-n_{\alpha}}) \sum_{w \in W} \varepsilon(w)e^{w(\rho^L)-\rho} \\
& = (e^{-\rho+\rho^L}\prod_{\alpha >0} (1-e^{-\alpha})^{-n_{\alpha}})\sum_{w\in W}\varepsilon(w)e^{w(\rho^L)-\rho^L}.
\end{split}
\]
Applying Weyl's denominator formula for $(X^*,Y^*,I)$ to the sum in the last expression, we find that
\[
\chi(\rho^L-\rho) = (e^{\rho^L-\rho}\prod_{\alpha >0} (1-e^{-\alpha})^{-n_{\alpha}})\prod_{\alpha^*>0} (1-e^{-\alpha^*})^{n^*_{\alpha}}.
\]
The statement of the Lemma follows immediately.
\end{proof}

\begin{remark}
In the finite or affine case, Lemma \ref{scaling} shows that $\Phi$ and $\Phi^*$ are in bijection $\alpha \leftrightarrow \alpha^*$, so that $\alpha^* = l_\alpha \alpha$ for some $l_\alpha \in \mathbb Z$ positive. Since all root spaces are one-dimensional in the finite case, we get a simple expression for the Weyl character of $\rho^L-\rho$:
\[
\chi(\rho^L-\rho) = e^{\rho^L-\rho}\prod_{\alpha \in \Phi, \alpha >0}(1 + e^{-\alpha} + \ldots + e^{-(l_\alpha-1)\alpha}).
\]
In the affine case suppose the generalized Cartan matrix is of type $X_{m}^{(s)}$ in the classification given in \cite[Chapter 4]{Kac}. Then although the root spaces corresponding to real roots are again all one-dimensional, the root space of weight $j\delta$ has dimension $|I|-1$ if $s$ divides $j$ and dimension $(m-|I|+1)/(s-1)$ otherwise. The explicit formula for $\chi(\rho^L-\rho)$ is thus similar but contains a more elaborate product contribution coming from imaginary roots.
\end{remark}

\begin{example}
Let $\ell=2$, and $\Ud_\ell$ be the quantum group of type $B_2$, so that $\Ud^*_\ell$ is of type $C_2$. We take $\alpha_1$ to be the long root and $\alpha_2$ to be the short root, so that $\rho^L-\rho = \varpi_2$. Setting $\mu = \varpi_1 \in X^*$, for example. and writing $y_i = e^{\varpi_i}$, we have
\[
\chi(\rho^L-\rho +\mu) = \chi(\rho) = y_1y_2(1+y_1^{-2}y_2^{-2})(1+y_1y_2^{-2})(1+y_1^{-1})(1-y_2^{-2})
\]
(since we always have $\chi(\rho) = e^{\rho}\prod_{\alpha >0}(1+e^{-\alpha})$). Now the representation of highest weight $\varpi_2$ is $4$-dimensional with character $\chi(\varpi_2) = y_2(1+y_1y_2^{-2})(1+y_1^{-1})$, and so dually we have $\chi^L(\varpi_1) = y_1(1+y_1^{-2}y_2^{-2})(1+y_2^{-2})$. The product formula of the previous Lemma now follows immediately. 
\end{example}

For $\nu_1,\nu_2, \nu_3 \in X^+$ let $\{c_{\nu_1,\nu_2}^{\nu_3}\}$ be the structure constants for the multiplication on the Grothendieck group of the category $\Oint(\mathfrak g)$ given by tensor product. Since $K_0(\Oint(\mathfrak g))$ injects into $\mathcal E$ via the character map we have
\[
\chi(\nu_1)\chi(\nu_2) =\sum_{\nu_3\in X^+} c_{\nu_1,\nu_2}^{\nu_3} \chi(\nu_3).
\]

\begin{theorem}
\label{branchingrule}
The Langlands duality branching rules $m_\mu^\lambda$ are positive. More precisely, we have
\[
m_\mu^\lambda = c_{\rho^L-\rho,\lambda}^{\mu+\rho^L-\rho}.
\]
\end{theorem}
\begin{proof}
Suppose we have $\Pi_\ell(\chi(\lambda)) = \sum_{\mu \in X^+\cap X^*, \mu \leq \lambda} m_{\mu}^\lambda \chi^L(\mu)$. Then it follows from Lemma \ref{combinatorialSt} that 
\begin{equation}
\label{Pipart}
\chi(\rho^L-\rho)\Pi_\ell(\chi(\lambda)) = \sum_{\mu \in X^+\cap X^*, \mu \leq \lambda} m_\mu^\lambda \chi(\mu+ \rho^L-\rho).
\end{equation}
On the other hand, since $\chi(\lambda)$ is $W$-invariant, and $\Pi_\ell$ is $W$-equivariant, we may apply Lemma \ref{combinatorialavoidance} with $\xi = (\chi(\lambda) - \Pi_\ell(\chi(\lambda))$ to obtain 
\begin{equation}
\label{notPipart}
\chi(\rho^L-\rho).(\chi(\lambda) - \Pi_\ell(\chi(\lambda))) = \sum_{\substack{\nu \in X^+ \\ \nu \notin \rho^L-\rho +X^*}} n_\nu^\lambda \chi(\nu).
\end{equation}
for some integers $n_\nu^\lambda \in \mathbb Z$. 
Finally, since we have 
\[
\chi(\rho^L-\rho).\chi(\lambda) = \sum_{\eta \in X^+} c_{\rho^L-\rho,\lambda}^\eta \chi(\eta),
\]
and the Weyl characters which can occur on the right-hand sides of equations (\ref{Pipart}) and (\ref{notPipart}) have highest weight which lie in different $X^*$-cosets, the claim immediately follows. 
\end{proof}

\begin{remark}
Note that the above argument shows that $\chi(\rho^L-\rho).(\chi(\lambda) - \Pi(\chi(\lambda))$ is indeed a positive sum of Weyl characters, in contrast to the general situation of Lemma \ref{combinatorialavoidance}. We will see in the next section that, at least in the finite-type case, it is the character of a direct summand of the $\Ud_\ell$-module $\nabla(\rho^L-\rho)\otimes \nabla(\lambda)$.

\end{remark}

Tensor product multiplicities have been computed combinatorially by various people: for finite type, building on a combinatorial description due to Lusztig, Berenstein and Zelevinsky have give ``polyhedral'' multiplicity formulas in \cite{BZ}; for a general Kac-Moody algebra, there is a Littlewood-Richardson rule in terms of Littelmann paths \cite{Li95}. 

\begin{example}
Consider again the case of $B_2$. We need to calculate the multiplicities in the tensor products $\nabla(\lambda)\otimes \nabla(\varpi_2)$ for $\lambda \in X^*\cap X^+$ (where $\alpha_2$ is the short root and $\alpha_1$ the long root). As we have seen above, the set of weights of $\nabla(\varpi_2)$ is the Weyl group orbit of $\varpi_2$ and each weight has multiplicity one. Let $W_{\varpi_2}$ be the stabilizer of $\varpi_2$ in $W$. Using the formula in the statement of Lemma \ref{combinatorialavoidance}, it follows that for $w \in W/W_{\varpi_2}$, the simple highest weight representation $\nabla(\lambda+w(\varpi_2))$ occurs exactly once in the tensor product, provided $\lambda+w(\varpi_2)$ is dominant (since it is easy to check that if $\lambda +w(\varpi_2)$ is not dominant, then $\lambda + w(\varpi_2) + \rho$ is not regular) and these are all the constituents. Hence we have
\[
\Pi_2(\chi(\lambda)) = \sum_{\substack{w \in W/W_{\varpi_2}\\ \lambda + w(\varpi_2) \in X^+}} \chi^L(\lambda + w(\varpi_2) - \varpi_2).
\]
Note that since $\lambda \in X^*\cap X^+$, the weight $\lambda + w(\varpi_2)$ is dominant if and only if $\lambda+w(\varpi_2)-\varpi_2$ is dominant. This recovers the calculations of \cite[Remark 6.9]{FH}.
\end{example}

\section{On Langlands duality branching rules in the finite-type case and tilting modules.}
\label{branchingtilting}

In this section we study the branching multiplicities $m_\mu^\lambda$ from a representation-theoretic point of view, and give an interpretation of the results of the last section using tilting modules. We restrict ourselves to the case of finite type algebras, as we will use the machinery of induction \textit{etc}. for quantum algebras at roots of unity provided by \cite{APW},\cite{Ka}, and the infinitesimal quantum groups defined by Lusztig. We begin by recalling the results on quantum groups at a root of unity that we will need. Since we work here with only finite type quantum groups we may work with the category $\mathcal F$ of finite dimensional representations (of type $1$).

In \cite{L90}, Lusztig defined root vectors $\theta_\alpha$ for each positive root $\alpha$, via a braid group action. For a positive root $\alpha$, let $\ell_\alpha = \ell_i$, where $\alpha$ is conjugate to the simple root $\alpha_i$ under the action of $W$ the Weyl group. 

\begin{prop}
Let $\mathfrak f$ be the subalgebra of $\mathbf f_\ell = \mathcal A_\ell \otimes_\mathcal A \mathbf f$ generated by $\{\theta_\alpha: \ell_\alpha \geq 2\}$. Then $\mathfrak f$ is a finite dimensional algebra and we have an isomorphism
\[
\mathbf f_\ell^* \otimes \mathfrak f \to \mathbf f_\ell,
\]
given by $(x,y) \mapsto Fr'(x).y$. 
\end{prop}
\begin{proof}
This is established in most cases in \cite[35]{L93}, except when $\ell$ is small. The excluded cases (which are already stated in \cite{L93} but without detailed proof) have been checked in \cite[2.7]{Ka}.
\end{proof}

\begin{definition}
We need various subalgebras of $\Ud_\ell$ and $\Ud_\ell^*$. Define algebras $\dot{\mathfrak u}$ and $\hat{\mathfrak u}$ by setting $\dot{\mathfrak u} = \{x^+1_\lambda y^-: x, y \in \mathfrak f, \lambda \in X\}$ and $\hat{\mathfrak u} = \{x1_\lambda: x \in \dot{\mathfrak u} \text{ or } x \in U^-\}$, (note that these are indeed subalgebras). Finally, we need the subalgebra $\Ud^{\leq 0}_\ell = \{x^-1_\lambda: x \in \mathbf f\}$. 
\end{definition}

In \cite{APW92} the authors define (derived) induction functors on integrable modules for the algebras $\dot{\mathfrak u}, \hat{\mathfrak u}, \Ud_\ell^{\leq 0}$ and $\Ud_\ell$ denoted $H^i(U_1/U_2, -)$ where $U_1 \supset U_2$ are two of the algebras above\footnote{In fact they work with ``unmodified'' algebras, but the categories of modules obviously correspond to categories of modules for the modified algebras -- see \cite{Ka} for more details.}. Given $\lambda \in X$, there is a natural one-dimensional module for $\Ud^{\leq 0}$ which we denote simply by $\lambda$. When $\lambda$ is dominant, $H^0(\Ud/\Ud^{\leq}, \lambda)$, just as in the classical theory of induction from a Borel subgroup, is an indecomposable module with character given by Weyl's formula, which we denote by $\nabla(\lambda)$. It has a unique simple submodule $L(\lambda)$. The dual of $\nabla(\lambda)$, denoted $\Delta(\lambda)$, is a costandard, or Weyl, module. The same theory exists for the algebra $\Ud^*_\ell$, though here of course the theory is easier because the category of $\Ud^*_\ell$-modules is semisimple. We will write $\nabla^*(\lambda)$, $\Delta^*(\lambda)$ for the corresponding modules for $\Ud^*_\ell$.

We define $\sigma_\ell = \sum_{i \in I} (\ell_i - 1)\varpi_i$, and let $St_\ell$, the Steinberg representation, be the module $\nabla(\sigma_\ell)$. The functor $H^0(\hat{\mathfrak u}/\Ud^{\leq 0}, -)$ is exact (see \cite[2.9]{Ka}), and we denote it by $\hat{Z}$. It is known that $St_\ell \cong \hat{Z}(\sigma_\ell)$ as $\hat{\mathfrak u}$-modules, and in fact $St_\ell$ is simple as both a $\Ud_\ell$ and  $\hat{\mathfrak u}$-module, see for example \cite[\S 0.9]{APW92} for more details.

We will also need the class of modules known as tilting modules, whose definition we now recall.
\begin{definition}
A $\Ud_\ell$ module is said to be \textit{tilting} if it has a filtration both by standard and costandard modules.
\end{definition}
We now review some of the basic results on tilting modules. Although all the results are standard, we sketch the proof to point out that they all hold even for small values of $\ell$.

\begin{theorem}
\label{tiltingthm}
\begin{enumerate}
\item
If $M_1,M_2$ are tilting modules, then so is $M_1\otimes M_2$.
\item If $M$ and $N$ are tilting modules, then $M \cong N$ if and only if $M$ and $N$ have the same character.
\end{enumerate}
\end{theorem}
\begin{proof}
The key to $(1)$ is to show that the tensor product of two standard modules has a filtration by standard modules. This follows even integrally from Lusztig's theory of based modules: see for example \cite{Ka98}. The general construction of tilting modules shows that for each $\lambda \in X^+$ there is a unique indecomposable tilting module $T(\lambda)$, where $\lambda$ occurs as a weight of $T(\lambda)$ with multiplicity one and all other weights of $T(\lambda)$ are strictly less than $\lambda$ . Moreover every indecomposable tilting module has this form. See \cite[\S 2]{A92} for more details. This readily implies $(2)$. 
\end{proof}

We also need to understand some relations between pulling back via the Frobenius, and induction. The main result of \cite{Ka} asserts that for $M$ a $\Ud_\ell^{*\leq 0}$-module there is an isomorphism
\begin{equation}
\label{isomorphism}
H^i(\Ud/\hat{\mathfrak u}, M^{Fr}) \cong H^i(\Ud^*_\ell/\Ud_\ell^{*\leq 0}, M)^{Fr}, \quad (i \geq 0).
\end{equation}
where $M^{Fr}$ is the $\hat{\mathfrak u}$-module obtained via composition with $Fr$, and similarly for the right-hand side. (This result is already established in \cite{APW92} with some restrictions). 

\begin{theorem}
Let $\lambda \in X^+\cap X^*$, then we have
\begin{equation}
\label{SteinbergRel}
H^i(\Ud_\ell/\Ud_\ell^{\leq 0}, \lambda + \sigma_\ell) \cong St_\ell \otimes H^i(\Ud^*_\ell/U^{*\leq 0}_\ell, \lambda)^{Fr}.
\end{equation}
\end{theorem}
\begin{proof}
With the ingredients provided by \cite{Ka}, the proof is standard. By (\ref{isomorphism}) we have
\[
H^i(\Ud^*_\ell/U^{*\leq 0}_\ell, \lambda)^{Fr} \cong H^i(\Ud/\hat{\mathfrak u}, \lambda^{Fr})
\]
Thus tensoring both sides with $St_\ell$ we find the right-hand side of (\ref{SteinbergRel}) is isomorphic to 
\[
\begin{split}
H^i(\Ud_\ell/\hat{\mathfrak u}, \lambda^{Fr})\otimes St_\ell &\cong H^i(\Ud_\ell/\hat{\mathfrak u}, \lambda^{Fr}\otimes St_\ell) \\
& \cong H^i(\Ud_\ell/\hat{\mathfrak u}, \lambda^{Fr}\otimes \hat{Z}(\sigma_\ell)) \\
& \cong  H^i(\Ud_\ell/\hat{\mathfrak u},  \hat{Z}( \lambda + \sigma_\ell))\\
& \cong H^i(\Ud_\ell/\Ud^{\leq 0}_\ell, \lambda + \sigma_\ell),
\end{split}
\]
where in the first line we use the tensor identity for $\Ud_\ell$-modules, in the second the isomorphism $St_\ell \cong \hat{Z}(\sigma_\ell)$, in the third the tensor identity for $\hat{\mathfrak u}$-modules, and in the final line, the fact that $\hat{Z}$ is exact, so the spectral sequence for the composition of induction functors degenerates to give an isomorphism $H^i(\Ud_\ell/\hat{\mathfrak u}, \hat{Z}(M)) \cong H^i(\Ud_\ell/\Ud_\ell^{\leq 0}, M)$. 
\end{proof}

\begin{remark}
The characteristic $p$ version of this theorem, due to Andersen, gives a short proof of Kempf's vanishing theorem.  Moreover, taking characters of $H^0$ when $\ell = d$ we recover Lemma \ref{combinatorialSt}, and thus it can be seen as the representation-theoretic version of that calculation. 
\end{remark}

\begin{prop}[\cite{APW},\cite{A92}]
\label{steinbergprop}
We have the following properties of the Steinberg module $St_\ell$.
\begin{enumerate}
\item
$St_\ell$ is injective and projective in $\mathcal F$.
\item
If $M$ is a finite dimensional representation, then $St_\ell \otimes M$ is tilting, projective and injective in $\mathcal F$.
\end{enumerate}
\end{prop}
\begin{proof}
We outline a proof of this theorem here to emphasize that it holds for all $\ell$ (at least over $\mathbb C$, which is the only field we need here). We must show that $\text{Ext}^1(St_\ell, L(\lambda)) = 0$ for all $\lambda \in X^+$. The linkage principle already implies that this Ext vanishes unless $\lambda = \sigma_\ell + \mu$ where $\mu \in X^*$. Now the previous theorem shows that these modules are in fact standard modules $\nabla(\sigma_\ell + \mu) = St_\ell\otimes Fr^*(\nabla^*(\mu))$. Hence it is enough to show that $\text{Ext}^1(St_\ell, \nabla(\sigma_\ell + \mu)) = 0$. 

Since $St_\ell$ is self-dual, this follows if we can show $\text{Ext}^1(\Delta(\sigma_\ell),\nabla(\sigma_\ell+\mu)) = 0$, but it is known (and crucial in the construction of tilting modules) that
\[
\text{Ext}^1(\Delta(\lambda), \nabla(\mu))=0, \quad  \forall \lambda, \mu \in X^+,
\]
and so we are done.

Self-duality also immediately implies that $St_\ell$ is injective. Moreover, using standard properties of $\text{Hom}$ and tensor product, it follows readily that $St_\ell \otimes E$ is injective and projective for any finite-dimensional module $E$. To see that it is tilting, one can show that any module can be imbedded in a module of the form $St_\ell \otimes T$ where $T$ is tilting. Since $St_\ell$ is also tilting, and tilting modules are closed under direct summands, it follows that indecomposable projectives and injective modules are tilting. See \cite{A92} for more details.
\end{proof}

\begin{cor}
Let $\mu \in X^*\cap X^+$. Then $\nabla(\mu+\rho^L-\rho)$ is simple, tilting, projective and injective.
\end{cor}
\begin{proof}
From Theorem \ref{SteinbergRel} and Lusztig's quantum version of Steinberg's tensor product theorem, it follows that the modules $\nabla(\mu+\rho^L-\rho) = \nabla(\rho^L-\rho)\otimes Fr^*(\nabla^*(\mu))$ are simple. By the previous Proposition, they are also tilting and injective. By duality, they are also projective.\end{proof}

We now examine the Langlands branching multiplicities. We would like a representation-theoretic interpretation of the calculation of these multiplicities in terms of tensor-product multiplicities in Theorem \ref{branchingrule}. The key, unsurprisingly, is Theorem \ref{SteinbergRel} and the theory of tilting modules outlined above. Notice first that $c^*(\nabla(\lambda))$ is a representation of $\Ud^*_\ell$, so it does not make sense to compare it to $\nabla(\lambda)$, however we may pull it back via $Fr$ without changing its character. Unfortunately, $Fr^*c^*(\nabla)$ still has no obvious (at least to the authors) relation to $\nabla(\lambda)$. Nevertheless once we tensor with the Steinberg representation, a natural relation appears. Recall from \cite{A03} that the linkage principle for $\Ud_\ell$ shows that the orbits of the $\rho$-shifted action of the affine Weyl group $\hat{W}$ are unions of blocks for $\Ud_\ell$. The simple modules $\nabla(\mu+ \rho^L-\rho)$ for $\mu \in X^*$  thus all lie in the union of blocks given by the orbits of $\hat{W}$ on $\rho^L-\rho + X^*$.

\begin{prop}
\label{occurence}
Let $\lambda \in X^*$. The module $St_\ell \otimes Fr^*c^*(\nabla(\lambda))$ is a direct summand of the module $St_\ell \otimes \nabla(\lambda)$, and moreover is precisely the summand which lies in the union of the blocks of $\Ud_\ell$ contained in the $\hat{W}$-orbits of $\rho^L-\rho + X^*$.
\end{prop}
\begin{proof}
By part $(2)$ of Proposition \ref{steinbergprop} we see that $St_\ell \otimes \nabla(\lambda)$ is a tilting module. Thus if $T(\gamma)$ denotes the indecomposable tilting module with highest weight $\gamma$, we may write 
\[
St_\ell \otimes \nabla(\lambda) = \bigoplus_{\nu \in X^+} T(\nu+\rho^L-\rho),
\]
a sum of indecomposable tilting modules (the tilting modules which occur must have highest weight of the form $\nu+\rho^L-\rho$, by \cite[5.12]{A92}). For any $\mu \in X^*$, Theorem \ref{SteinbergRel} combined with Proposition \ref{steinbergprop} shows that $\nabla(\mu+\rho^L-\rho)$ is projective and injective and tilting, thus it cannot occur as a composition factor of a standard filtration of $T(\nu+\rho^L-\rho)$ for $\nu \notin X^*$. Therefore
\[
St_\ell\otimes \nabla(\lambda) = T \oplus \big(\bigoplus_{\mu \in X^*}\nabla(\mu+\rho^L-\rho)^{\oplus c_{\rho^L-\rho, \mu}^{\mu+\rho^L-\rho}}\big),
\]
where $T$ is a tilting module whose character lies entirely in the (positive) span of the Weyl characters $\chi(\nu+\rho^L-\rho)$ for $\nu \notin X^*$.

On the other hand,  we have 
\[
St_\ell \otimes Fr^*c^*(\nabla(\lambda)) = \bigoplus_{\mu \in X^*} \nabla(\mu+ \rho^L-\rho)^{\oplus m_\mu^\lambda},
\]
Hence using Theorem \ref{branchingrule} the result follows.
\end{proof}

\begin{remark}
This also shows that the expression $\chi(\rho^L-\rho).(\chi(\lambda) -\Pi(\chi(\lambda))$ is the character of $T$ in the above proof, which is also a direct summand of $St_\ell\otimes \nabla(\lambda)$. Note also that the above proof shows that $m_\mu^\lambda \leq c_{\mu,\rho^L-\rho}^{\mu+\rho^L-\rho}$, independently of Section \ref{combinatorialbranching} since tilting modules are determined by their character. It would be interesting to know if there is a natural $\Ud$-module map between $St_\ell \otimes \nabla(\lambda)$ and $St_\ell \otimes Fr^*c^*(\nabla(\lambda))$.
\end{remark}

\end{document}